\documentclass[times,11pt]{article}

\usepackage[margin=3cm]{geometry}
\usepackage{moreverb}

\usepackage[colorlinks,bookmarksopen,bookmarksnumbered,citecolor=red,urlcolor=red]{hyperref}
\title{Hermitian Preconditioning for a class of Non-Hermitian Linear Systems}

\author{Nicole Spillane \thanks{CNRS, CMAP, Ecole Polytechnique, Institut Polytechnique de Paris, 91128 Palaiseau Cedex, France (\textit{nicole.spillane@cmap.polytechnique.fr})}}

\usepackage{algorithm}
\usepackage{algorithmic}
\usepackage{amsthm}
\usepackage{amsmath}
\usepackage{amssymb}
\usepackage{tikz,pgfplots}
\usepackage{subcaption}
\usepackage{multirow}
\usepackage{enumerate}
\usepackage{enumitem}
\usepackage{stmaryrd}

\newcommand{\rhs}{\ensuremath{\mathbf{b}}}

\newcommand{\matid}{\ensuremath{\mathbf{{I}}}}
\newcommand{\bu}{\ensuremath{\mathbf{u}}}
\newcommand{\by}{\ensuremath{\mathbf{y}}}
\newcommand{\bR}{\ensuremath{\mathbf{R}}}
\newcommand{\bP}{\ensuremath{\mathbf{P}}}
\newcommand{\bB}{\ensuremath{\mathbf{B}}}

\newcommand{\bM}{\ensuremath{\mathbf{M}}}
\newcommand{\bN}{\ensuremath{\mathbf{N}}}
\newcommand{\bW}{\ensuremath{\mathbf{W}}}

\newcommand{\bPi}{\ensuremath{\boldsymbol{\Pi}}}

\newcommand{\bx}{\ensuremath{\mathbf{x}}}
\newcommand{\bb}{\ensuremath{\mathbf{b}}}
\newcommand{\bK}{\ensuremath{\mathbf{K}}}
\newcommand{\bA}{\ensuremath{\mathbf{A}}}
\newcommand{\bH}{\ensuremath{\mathbf{H}}}
\newcommand{\bS}{\ensuremath{\mathbf{S}}}
\newcommand{\br}{\ensuremath{\mathbf{r}}}

\newcommand{\bz}{\ensuremath{\mathbf{z}}}
\newcommand{\bw}{\ensuremath{\mathbf{w}}}
\newcommand{\bq}{\ensuremath{\mathbf{q}}}

\newcommand{\bp}{\ensuremath{\mathbf{p}}}

\newcommand{\0}{\ensuremath{^{0}}}
\newcommand{\s}{\ensuremath{^{s}}}
\newcommand{\conj}{\ensuremath{^{*}}}

\usepackage{stmaryrd}

\graphicspath{{Figures.d/}}

\theoremstyle{plain}
\newtheorem{theorem}{Theorem}

\newtheorem{corollary}{Corollary}
\newtheorem{lemma}{Lemma}

\newtheorem{remark}{Remark} 

\begin{document}

\maketitle

\pagestyle{myheadings}
\thispagestyle{plain}
\markboth{Nicole SPILLANE}{Hermitian preconditioning for non-Hermitian systems}


\textbf{MSCodes: }  65F10, 65Y05, 68W40

\abstract{
This work considers the convergence of GMRES for non-singular problems. GMRES is interpreted as the GCR method which allows for simple proofs of the convergence estimates. Preconditioning and weighted norms within GMRES are considered. The objective is to provide a way of choosing the preconditioner and GMRES norm that ensure fast convergence. The main focus of the article is on Hermitian preconditioning (even for non-Hermitian problems). It is proposed to choose a Hermitian preconditioner $\bH$ and to apply GMRES in the inner product induced by $\bH$. If moreover, the problem matrix $\bA$ is positive definite, then a new convergence bound is proved that depends only on how well $\bH$ preconditions the Hermitian part of $\bA$, and on how non-Hermitian $\bA$ is. In particular, if a scalable preconditioner is known for the Hermitian part of $\bA$, then the proposed method is also scalable. This result is illustrated numerically. 
}

{\footnotesize
\tableofcontents
}
\section{Introduction}

GMRES, or the Generalized Minimal Residual Method, is a method of choice for solving general linear systems. First introduced by \cite{zbMATH03967793}, the convergence of GMRES has since been extensively studied \cite{zbMATH03831185,elman1982iterative,zbMATH06385506,zbMATH05029264,zbMATH06394941,zbMATH00089377}. In this manuscript, linear systems 
\begin{equation}
 \bA {\bx} = \bb,
\end{equation}
 are considered. Initially, the only assumption is that $\bA$ is a general non-singular matrix over the field $\mathbb K$ with $\mathbb K = \mathbb R \text{ or } \mathbb C$. The focus then shifts to Hermitian preconditioning and matrices whose Hermitian part is positive definite. In each of these cases, convergence is examined for the weighted GMRES algorithm, a version of GMRES where the Euclidean inner product has been replaced by $(\bx, \by) \mapsto \langle \bW \bx, \by \rangle$ with $\bW$ Hermitian positive definite. 
The objective of the present work is to prove a convergence bound that can then be used to choose the preconditioner $\bH$ and the weight matrix $\bW$ in a smart way.

Convergence of GMRES in the Euclidean inner product when $\bA$ is positive definite is studied in \cite{zbMATH03831185}. Interestingly, \cite{zbMATH03831185} precedes the introduction of GMRES in  \cite{zbMATH03967793} because the results are for the Generalized Conjugate Residual algorithm, or GCR (that produces the same iterates as GMRES).

In his PhD thesis \cite{cai1989some}, Cai proposes to select for GMRES applied to a matrix $\bP$, an inner product that `is chosen to take advantage of some special properties of $\bP$'. In collaboration with Widlund \cite{zbMATH00036024}, they propose domain decomposition preconditioners for non-symmetric and indefinite second order PDEs. GMRES is considered in the norm induced by the highest order term in the variational form: the energy norm. More will be said about these results below. 
{Another choice that has been proposed in \cite{zbMATH01096035} is to apply GMRES in the inner product induced by the inverse of a symmetric positive definite (spd) preconditioner. Later, a different set of authors justify this choice in \cite{zbMATH01201042} by the argument that, if $\bH$ is spd and $\bA$ is nearly symmetric then $\bH \bA$ is nearly $\bH^{-1}$-self adjoint. In \cite{zbMATH01271905}, preconditioning of saddle point problems is tackled and the bounds from \cite{zbMATH01096035} are applied in an inner product that is derived from a non-symmetric triangular preconditioner.  The authors of \cite{zbMATH06290704} address the problem of finding inner products that make a general preconditioned system nearly normal or nearly non-symmetric. A detailed presentation of what is now called weighted GMRES, with results both on convergence and implementation strategy, can be found in \cite{zbMATH05626642}. In \cite{zbMATH01268410}, the author reinvents the idea of changing the norm within GMRES in a general framework and coins the term \textit{weighted GMRES}. The inner products considered by \cite{zbMATH01268410} are diagonal matrices of weights that changes at each restart of the method. Larger weights are associated to the larger components of the residual at the end of the previous cycle leading to a faster convergence. This idea is compared to other accelerators for GMRES in \cite{zbMATH06376430}.    
}

The method that is proposed in Section~\ref{sec:PD} assumes that an efficient preconditioner $\bH$ for the Hermitian part of $\bA$ (\textit{i.e.,} $1/2 (\bA + \bA^*)$) is known. The same $\bH$ is applied within GMRES to solve the non-Hermitian problem (for $\bA$). As an illustration, the convection-diffusion-reaction discretized by finite elements is solved at the end of this article. The preconditioner $\bH$ is chosen to be a two-level domain decomposition preconditioner \cite{zbMATH02113718} with the GenEO coarse space introduced by \cite{2011SpillaneCR,spillane2013abstract} (see also \cite{SPILLANE:2013:FETI_GenEO_IJNME,haferssas2017additive,klawonn2016adaptive,calvo2016adaptive,pechstein2017unified,arxiv.2104.00280} for a non exhaustive list of extensions and related work).  
It has been shown, \textit{e.g.}, in \cite{6877513}, that GenEO is a very powerful and scalable solver for spd problems.

The idea of separating the Hermitian and skew-Hermitian part of a matrix in order to approximate the solution of a linear system has been exploited in the Hermitian and skew-{Hermitian} splitting methods introduced in \cite{zbMATH02027915}.  The field of domain decomposition for non-spd problems was paved by \cite{cai1989some,zbMATH00036024} (see also \cite{zbMATH02113718}[Chapter 11]). The authors solve convection-diffusion-reaction with a two-level additive Schwarz preconditioner where the coarse space is {based on a} coarse grid. The one level component in the preconditioner takes two forms: either the local solvers come from the original (non-symmetric and/or indefinite) matrix, or the local solvers come from an spd part of the problem matrix. GMRES is applied in the norm induced by that spd part. In both cases the coarse projector is for the original operator. If the coarse mesh is fine enough, the rate of convergence is shown to be independent of the number of degrees of freedom and the number of local problems (scalability). A simplified explanation for the presence of a condition on the size of the subdomains is that, for the theory to go through, the second order term must dominate the other terms. It is known that the first non-zero eigenvalue of $-\Delta$ on regular subdomains of diameter $H$ is of the order of $1/H$.  The algorithm is generalized in \cite{zbMATH00149253} to any matrix that can be viewed as a perturbation of an spd matrix. Their proposed preconditioner is the combination of a very good preconditioner for the spd part and a coarse solve. A parameter $\delta_0$ qualifies how efficient the coarse space is at filling the gap between the original problem and the spd one. It enters into the convergence estimate. The algorithm is later called  CSPD for Coarse Grid Plus SPD Preconditioning in a numerical comparison with other algorithms \cite{cai1992comparison}. 

More recently, domain decomposition for Helmholtz has been studied \cite{zbMATH06713479,zbMATH07248609}. This case is symmetric indefinite and applying GMRES in the energy norm is a crucial part of the proof. Weighted GMRES was also applied to study the spd GenEO eigenproblem applied to indefinite and non-self-adjoint problems in \cite{arxiv.2103.16703,zbMATH07726053}. It is proved and observed numerically that GenEO performs well also on a family of non-spd second-order problems. Finally, the authors of \cite{zbMATH07395831} prove an abstract framework for one-level additive Schwarz for non-Hermitian or indefinite problems. They illustrate their results by solving the convection-diffusion-reaction equation. The present work takes a more algebraic route and proves results that are not restricted to domain decomposition.  

The outline for the rest of the article is as follows. In Section~\ref{sec:GMRES}, some notation is introduced and Theorem~\ref{th:general} gives an overview of some of the convergence results { and their connection to existing results}. In Section~\ref{sec:Orthomin}, the GMRES algorithm is studied through the study of an equivalent form that is the GCR algorithm. A convergence bound is proved in Theorem~\ref{th:RPconv} that is connected to field of value, or Elman, estimate \cite{zbMATH03831185,elman1982iterative,zbMATH05029264}. In Section~\ref{sec:SymPrec}, Theorem~\ref{th:HPconv}, a special case is considered where the preconditioner $\bH$ is Hermitian and GMRES/GCR is applied in the $\bH$ inner product (or $\bH^{-1}$ for left preconditioned GMRES). 
For positive-definite $\bA$, a final convergence estimate is proved in Theorem~\ref{th:final} that makes explicit the rate at which the non-Hermitianness of $\bA$ slows down convergence. As an illustration of this result, in Section~\ref{sec:Numerical},  a solver is proposed for the convection-diffusion-reaction equation. The preconditioner is a domain decomposition preconditioner with a GenEO spectral coarse space. This way, the Hermitian part of the problem is very well preconditioned. It is shown theoretically that the overall convergence does not depend on the number of subdomains (scalability), or on the discretization step. Numerical experiments show that Hermitian preconditioning can be very efficient and scalable for mildly non-Hermitian problems.  

\section{Problem posed, notation and main results}
\label{sec:GMRES}

Let $\bK = \mathbb R$ or $\mathbb C$ be the field over which the linear system is considered. Let $\bA \in \mathbb K^{n\times n}$ be a non-singular matrix. Given any $\bb \in \mathbb K^n$, the problem at hand is to find ${\bx} \in \mathbb K^n$ such that:
\[
\bA {\bx} = \bb.
\]
The chosen methodology is to apply weighted and preconditioned (WP-) GMRES. Let $\bH \in \mathbb K^{n\times n}$ denote the preconditioner and $\bW \in \mathbb K^{n\times n}$ denote the weight matrix. It is assumed that $\bH$ is non-singular and that $\bW$ is Hermitian positive definite (hpd). The inner product and norm induced by $\bW$ are denoted by $\langle \cdot, \cdot \rangle_\bW$ and $\|\cdot \|_\bW$, respectively:
\[
\langle \bx, \by \rangle_\bW = \langle \bW \bx, \by \rangle =  \langle \bx,\bW \by \rangle = \by\conj \bW \bx \text{ and } {\|\bx\|_\bw = {\langle \bx, \bx \rangle_\bW}^{1/2}},\, \forall \, \bx,\,\by \in \mathbb K^n . 
\]

Any matrix $\bB \in \mathbb K^{n \times n}$  can be split into the sum of its Hermitian part and its skew-Hermitian part. The notation used is  
\begin{equation}
\label{eq:hsh}
\bM(\bB) = \frac{\bB +\bB^*}{2} \text{ (Hermitian part) and } \bN(\bB) = \frac{\bB - \bB^*}{2} \text{(skew-Hermitian part)}.
\end{equation}

An overview of some results in the article is given in the following theorem.

\begin{theorem}[Summary of main results]
\label{th:general}
Assume that the operator $\bA \in \mathbb K^{n\times n}$ and preconditioner $\bH \in \mathbb K^{n\times n}$ are non-singular, and that the weight matrix $\bW \in \mathbb K^{n\times n}$ is hpd. The $i$-th iterate of {weighted and preconditioned GMRES (WP-GMRES)} with right preconditioning satisfies
\begin{align*}
\label{eq:allconv}
\frac{ \|\br_{i} \|_\bW}{\|\br_{0} \|_\bW} & \leq \left[ 1 - \inf\limits_{\by\neq 0}    \frac{|\langle {\bA \bH \by}, \by \rangle_\bW|^2}{ \|\bA \bH \by \|_\bW^2  \|\by\|_\bW^2}\right]^{i/2}\\ 
& \leq 
 \left[ 1- \inf\limits_{\by\neq 0} \frac{|\langle \bM( \bA^{-1}) \by, \by \rangle|}{ \langle \bH \by, \by \rangle} \times \inf\limits_{\by\neq 0} \frac{|\langle \bM(\bA) \by, \by \rangle|}{ \langle \bH^{-1} \by, \by \rangle}\right]^{i/2} \text{ if $\bH = \bW$ is hpd}\\
&   \leq \left[ 1 -  \frac{1/\kappa(\bH \bM(\bA))}{1 + \rho(\bM(\bA)^{-1} \bN(\bA) )^2 } \right]^{i/2} \text{if $\bH = \bW$ is hpd and $\bA$ positive definite},
\end{align*}
where $\kappa(\bH \bM(\bA))$ is the condition number of $\bH \bM(\bA)$ and $\rho(\cdot)$ denotes the spectral radius of a matrix. For the last two estimates, it has been assumed that the preconditioner $\bH$ is hpd and that $\bW = \bH$. In this case, the algorithm will be referred to as WHP-GMRES for Weighted with a Hermitian Preconditioner. 
\end{theorem}
\begin{proof}
{The proof proceeds as follows. First, in Theorem~\ref{th:WPGCRintro}, the equivalence between weighted and preconditioned GCR, or WP-GCR, (Algorithm~\ref{alg:Worthoright}) and WP-GMRES is established as long as the origin is not in the field of values of the preconditioned operator. In the opposite case, the bounds above simplify to $\|\br_{i} \|_\bW \leq \|\br_{0} \|_\bW$, a trivial result. The first estimate in the theorem is proved for WP-GCR in Theorem~\ref{th:RPconv}. The second estimate in the theorem is proved for WHP-GCR (where again HP stands for Hermitian preconditioning) in Theorem~\ref{th:HPconv}. It is a direct consequence of the first when $\bH$ is spd and $\bW = \bH$. The last estimate in the theorem is proved in Corollary~\ref{cor:final} .}   
\end{proof}
{
\paragraph{Connection to previous results} Field of value estimates have previously been considered. 
The first result in the theorem generalizes \cite{zbMATH01226274}[Theorem 6.1] (and even more precisely the second last line in the proof). It was also noticed there, in \cite{zbMATH01226274}[Theorem 6.2], that the case $\bW = \bH$ spd is of particular interest and simplifies to the second estimate in the theorem (see also the earlier work \cite{zbMATH01096035}[Theorem 3.2]). These pioneering contributions consider problems arising from discretizations of elliptic and bounded variational problems. The present manuscript considers the problem directly in its algebraic form. The first two estimates can be informative even if the matrix does not arise from an elliptic PDE or even if it is not positive definite. 
There is also a connection between WHP-GMRES and the algorithm in Section 10.4 of \cite{elman1982iterative} where split preconditioning by $\bS^\top$ on the left and $\bS$ on the right is applied. The scope of split preconditioning is very restrictive but it could be proved that it is equivalent to WHP-GMRES with $\bH = \bS \bS^\top$. This would allow to extend the bound in Section 10.4 of \cite{elman1982iterative} to WHP-GMRES. Even then, the result would still be less sharp than the third bound in Theorem~\ref{th:general} because the condition number is squared.  
Finally, the present work generalizes the bounds to the complex case, where it no longer holds that $\langle \bA \by, \by \rangle = \langle \bM(\bA) \by, \by \rangle$. 
}

\section{Convergence of WP-GMRES viewed as WP-GCR} 
\label{sec:Orthomin}

GCR \cite{zbMATH03831185}, also known as Orthomin \cite{vinsome1976orthomin} is equivalent to GMRES in the sense that it generates the same approximate solutions at each iteration. GMRES is usually preferred as it is slightly less computationally expensive and more stable. However, GCR has the advantage of a simpler presentation and the proofs in this article all come from the GCR formulation of GMRES. 

\subsection{WP-GCR with right preconditioning}
\label{subs:WPGCRright}
Weighted and preconditioned GCR (WP-GCR) with right preconditioning is presented in Algorithm~\ref{alg:Worthoright}. The initial guess $\bx_0$ is assumed to be any vector in $\mathbb K^n$.  

\begin{algorithm}
\caption{WP-GCR with right preconditioning}
\label{alg:Worthoright}
\begin{algorithmic}
\REQUIRE $\bx_0 \in \mathbb R^n$
\STATE $\br_0= \bb - \bA \bx_{0}$ 
\STATE $\bz_0 = \bH \br_0$ 
\STATE $\bp_0 = \bz_0$
\STATE $\bq_0 = \bA \bp_{0}$
\FOR{$i = 0,\,1,\, \dots,\;$convergence}
\STATE  $\delta_i = \langle \bq_i, \bq_{i}\rangle_\bW $; \quad $\gamma_i = \langle {\bq_{i}}, \br_{i} \rangle_\bW$; \quad $\alpha_i =  \gamma_i/\delta_i $
\STATE  $\bx_{i+1} = \bx_{i}+ \alpha_i \bp_i$ 
\STATE  $\br_{i+1} = \br_i- \alpha_i \bq_i$ 
\STATE  ${\bz_{i+1}} =  \bH \br_{i+1}$
  \FOR{$j = 0,\,1,\, \dots,i$}
\STATE    $\Phi_{i,j} = \langle \bq_j, \bA \bz_{i+1} \rangle_\bW$; \quad  $\beta_{i,j} = \Phi_{i,j}/ \delta_j $
\ENDFOR
\STATE  $\bp_{i+1} = \bz_{i+1}   - \sum\limits_{j=0}^{i}\beta_{i,j} \bp_j  $\;
\STATE  $\bq_{i+1} = \bA \bz_{i+1}  - \sum\limits_{j=0}^{i} \beta_{i,j} \bq_j $\;
\ENDFOR
\RETURN Return {$\bx_{i+1}$}
\end{algorithmic}
\end{algorithm}

In the following theorem, it is proved that WP-GCR is indeed equivalent to WP-GMRES unless WP-GCR has an unlucky breakdown. This can only happen if $0$ is in the $\bW$-field of values of $\bA\bH$, a set defined by
\begin{equation}
\label{eq:defFOV}
 W_\bW(\bA\bH) = \left\{ \frac{\langle {\bA \bH \bu}, \bu \rangle_\bW}{ \langle \bu, \bu \rangle_\bW} ; \bu \in \mathbb C^n \setminus \{ 0 \}\right\}.
\end{equation}
Although the proof is not new it has been included. Indeed, some intermediary results in the proof are used in subsequent proofs.

{
\begin{theorem}
\label{th:WPGCRintro}
Assume that the operator $\bA \in \mathbb K^{n\times n}$ and preconditioner $\bH \in \mathbb K^{n\times n}$ are non-singular, and that the weight matrix $\bW \in \mathbb K^{n\times n}$ is hpd. The $i$-th residual of Algorithm~\ref{alg:Worthoright} satisfies  
\begin{equation}
\label{eq:RPmin}
\| \br_i \|_\bW = \operatorname{min}\left\{ \| \mathbf b - \bA \bx \|_{ \bW } ; \, {\bx\in\bx_{0} + \operatorname{span} \{\bp_0, \dots, \bp_{i-1}\}}\right\} . 
\end{equation}
Moreover, if $0 \not \in  W_\bW(\bA\bH)$, then  $ \operatorname{span} \{\bp_0, \dots, \bp_{i-1}\} = \mathcal K_i$ where 
\[ 
\mathcal K_i := \{ \bH \br_0, \bH \bA \bH \br_0, \dots,  (\bH \bA)^{i-1} \bH \br_0\}
\]
is the Krylov subspace, \textit{i.e.},  WP-GCR returns the same approximate solutions as WP-GMRES. 
\end{theorem}

\begin{proof}
The vectors $\bq_i =  \bA\bp_i$ are pairwise orthogonal, \textit{i.e.},
\begin{equation}
\label{eq:RPorth}
\langle \bq_j , \bq_i \rangle_\bW = 0 \text{ if } i \neq j.
\end{equation}
Indeed, by symmetry, it suffices to prove by recursion over $i \geq 0$ that: $\langle \bA  \bp_i,\bA \bp_j \rangle_\bW = 0$ for every $j < i$.
This is easy by recalling the definitions of $\bp_i$ and $\beta_{i,j}$. 

Next, the minimization property~\eqref{eq:RPmin} is proved. By an immediate recursion, it holds that $\br_i = \br_0 -  \sum\limits_{j=0}^{i-1} \gamma_j/\delta_j \bA \bp_j$. We notice that 
$\gamma_i = \langle \bq_i, \br_i \rangle_\bW =  \langle \bq_i, \br_0 \rangle_\bW -  \sum\limits_{j=0}^{i-1} \gamma_j/\delta_j \langle \bq_i , \bA \bp_j \rangle = \langle \bq_i, \br_0 \rangle_\bW$  
by~\eqref{eq:RPorth}. The $i$-th residual can now be rewritten as
\[
\br_0 - \br_i =  \sum\limits_{j=0}^{i-1} \frac{\langle \bA \bp_j, \br_0 \rangle_\bW}{\langle \bA \bp_j, \bA \bp_j \rangle_\bW} \bA \bp_j.
\] 
This means that $\br_0 - \br_i$ is the $\bW$-orthogonal projection of $\br_0$ onto $\operatorname{span}\{\bq_j; \, j <i \}$.  Two other equivalent characterizations of the orthogonal projection are  
\begin{equation}
\label{eq:RPprop}
\br_i\in\br_{0}  +  \operatorname{span}\{\bq_j; \, j <i \} \text{ with }  \br_i = (\br_0 - \br_i) - \br_0 \perp^\bW \operatorname{span}\{\bq_j; \, j <i \},
\end{equation}
and $\| \br_i \|_{\bW} = \operatorname{min}\left\{ \| \br \|_{\bW} ; \, {\br\in\br_{0}  + \operatorname{span}\{\bq_j; \, j <i \}} \right\}$. Then \eqref{eq:RPmin} follows by the change of variables $\br = \mathbf b - \bA\bx$ and $\br_0 =  \bb -  \bA \bx_0$.  

It remains to justify that $\operatorname{span}\{\bp_j; \, j <i \} = \mathcal K_i$. It is obvious that $\operatorname{span}\{\bp_j; \, j <i \} \subset  \mathcal K_i$. The reverse inclusion is true unless $\bA \bz_{i-1} \in \operatorname{span}\{\bq_j; \, j <i-1 \}$. Then, by \eqref{eq:RPprop}, $\br_{i-1} \perp^{\bW} \bA \bz_{i-1}$, \textit{i.e.,} $\langle \br_{i-1}, \bA \bH  \br_{i-1}  \rangle_\bW = 0 = \gamma_{i-1}$. If $0 \not \in  W_\bW(\bA\bH)$, this implies $\br_{i-1} = \mathbf 0$ which is a lucky breakdown of both GCR and GMRES. 
The proof ends by recalling the characterization of the $i$-th iterate of preconditioned GMRES in \textit{e.g.}, \cite{zbMATH01953444}[Section 9]. Following the idea in \cite{cai1989some,zbMATH01268410}, the Euclidean product can be changed to the $\bW$-inner product in order to get the result for WP-GMRES. 
\end{proof}

}
The following observations can be made:
\begin{itemize}
\item Preconditioning modifies the Krylov subspace. 
\item Weighting modifies which norm of the residual is minimized.
\item Weighting does not modify the Krylov subspace. 
\end{itemize}

The speed of convergence of the algorithm is addressed next by comparing two subsequent residuals. 
\begin{theorem}[Convergence of WP-GCR]
\label{th:RPconv}
Assume that the operator $\bA \in \mathbb K^{n\times n}$ and preconditioner $\bH \in \mathbb K^{n\times n}$ are non-singular, and that the weight matrix $\bW \in \mathbb K^{n\times n}$ is hpd. The $i$-th iterate of Algorithm~\ref{alg:Worthoright} satisfies
\begin{equation}
\label{eq:RPsuff}
\frac{ \|\br_{i} \|_\bW}{\|\br_{0} \|_\bW} \leq \left( 1 -  \inf\limits_{\by \neq \mathbf 0}  \frac{|\langle {\bA \bH \by}, \by \rangle_\bW|^2}{ \|\bA \bH \by \|_\bW^2  \|\by\|_\bW^2} \right)^{i/2}. 
\end{equation}
\end{theorem}
\begin{proof}

From the residual update formula we get $\br_{i} = \br_{i+1} + \alpha_i \bq_i$. {By the choice of $\alpha_i$,} $ \br_{i+1} \perp^\bW \bq_i$ so
\[
\|\br_{i} \|_\bW^2 = \|\br_{i+1} \|_\bW^2 + |\alpha_i|^2 \|\bq_i\|_\bW^2 =  \|\br_{i+1} \|_\bW^2 + \frac{|\langle {\bq_{i}}, \br_{i} \rangle_\bW^2|}{\|\bq_i\|_\bW^4 } \|\bq_i\|_\bW^2 . 
\]
The relative decrease in residual between two subsequent iterations is 
\begin{equation}
\label{eq:RPstep0}
\frac{ \|\br_{i+1} \|_\bW^2 }{\|\br_{i} \|_\bW^2} =  1 - \frac{|\langle {\bq_{i}}, \br_{i} \rangle_\bW|^2}{\|\bq_i\|_\bW^2 \|\br_{i} \|_\bW^2 }. 
\end{equation}
Taking the $\bW$-inner product of  $\bq_{i} = \bA \bz_{i} - \sum\limits_{j=0}^{i-1} \bq_j  \beta_{i,j}$ by $\br_i$ leads to 
\[
\langle \bq_{i}, \br_i \rangle_\bW  =\langle \bA \bz_{i}, \br_i \rangle_\bW  -  \sum\limits_{j=0}^{i-1} \beta_{i,j} \langle \bq_j, \br_i \rangle_\bW  = \langle \bA \bz_{i}, \br_i \rangle_\bW, 
\] 
and 
\begin{equation}
\label{eq:RPstep1}
\frac{ \|\br_{i+1} \|_\bW^2 }{\|\br_{i} \|_\bW^2} =  1 - \frac{|\langle {\bA \bz_{i}}, \br_{i} \rangle|_\bW^2}{\|\bq_i\|_\bW^2 \|\br_{i} \|_\bW^2 } =  1 - \frac{|\langle {\bA \bH \br_{i}}, \br_{i} \rangle_\bW|^2}{\|\bq_i\|_\bW^2 \|\br_{i} \|_\bW^2 }  . 
\end{equation}

Next, from the orthogonalisation formula and \eqref{eq:RPorth}, it is deduced that
\begin{equation}
\label{eq:diffwithMR}
 \|\bA \bz_{i} \|_\bW^2=  \| \bq_{i}  + \sum\limits_{j=0}^{i-1} \beta_{i,j} \bq_j \|_\bW^2 =  \| \bq_{i}\|_\bW^2  + \sum\limits_{j=0}^{i-1} |\beta_{i,j}|^2 \| \bq_j \|_\bW^2 \geq \|\bq_{i}\|_\bW^2.
\end{equation}
Finally, the decrease in residual between two subsequent iterations of Algorithm~\ref{alg:Worthoright} is bounded by 
\begin{equation}
\label{eq:Worthoright}
\frac{ \|\br_{i+1} \|_\bW}{\|\br_{i} \|_\bW} \leq \left[ 1 - \frac{|\langle {\bA \bz_{i}}, \br_{i} \rangle|_\bW^2}{ \|\bA \bz_{i} \|_\bW^2  \|\br_{i} \|_\bW^2 } \right]^{1/2} =  \left[ 1 - \frac{|\langle {\bA \bH \br_{i}}, \br_{i} \rangle_\bW|^2}{ \|\bA \bH \br_{i} \|_\bW^2  \|\br_{i} \|_\bW^2 } \right]^{1/2},  
\end{equation}
where the Cauchy-Schwarz inequality ensures that the square root is well defined. 
\end{proof}

Next, this is reformulated to match an often cited result (out of many) in \cite{zbMATH03831185}.
 
\begin{corollary}[Field of Value, or Elman, estimate]
\label{cor:Elman}
Under the assumptions of Theorem~\ref{th:RPconv}, 
the $i$-th iterate of Algorithm~\ref{alg:Worthoright} satisfies
\[
\frac{ \|\br_{i+1} \|_\bW}{\|\br_{0} \|_\bW} \leq \left[ 1 -  \frac{d(0,W_\bW(\bA\bH))^2}{\|\bA\bH \|_\bW^2}\right]^{i/2}, 
\]
where
\begin{itemize}
\item $ d(0,W_\bW(\bA\bH)) =  \inf \left\{ \frac{|\langle {\bA \bH \bu}, \bu \rangle_\bW|}{ \langle \bu, \bu \rangle_\bW} ; \bu \in \mathbb C^n \setminus\{0\} \right\}$ is the distance to zero of the $\bW$-field of values of $\bA \bH$ defined in \eqref{eq:defFOV},
\item ${\|\bA\bH \|_\bW}$ denotes the matrix norm of $\bA\bH$ induced by the vector norm $\bW$.
\end{itemize}
\end{corollary}
\begin{proof}
The terms in \eqref{eq:RPsuff} can be grouped as
\begin{equation}
 \left[\frac{\langle {\bA \bH \by}, \by \rangle_\bW}{  \|\by\|_\bW^2}\times \frac{  \|\by\|_\bW}{\|\bA \bH \by \|_\bW}  \right]^2 \geq \left[ \frac{d(0,W_\bW(\bA\bH))}{\|\bA\bH \|_\bW} \right]^2, 
\end{equation}
where the numerator minimizes the first term in the product and the denominator maximizes the inverse of the second. 
\end{proof}

The result of Theorem~\ref{th:RPconv} is stronger than the field of value estimate in Corollary~\ref{cor:Elman} {as discussed in \cite{arxiv12115969} and \cite{zbMATH05080488}.} Indeed, a bound for $C$ in Theorem~\ref{th:RPconv} can be found without necessarily bounding $d(0,W_\bW(\bA\bH))$ and $\|\bA\bH \|_\bW$ independently. {Another way of saying this is that the terms in $C$ can be grouped differently than in the field of value bound. This is done in the next section and was initially proposed by \cite{zbMATH01096035}.} 

\begin{remark}[Breakdown and equivalence with WP-GMRES]
If $0$ is in the $\bW$-field of values of $\bH \bA$, the right hand side of the Elman estimate is $1$ rendering it useless. In fact, the proof in \cite{zbMATH03831185} makes the assumption that the $\bA$ is positive-definite so the case where $0$ is in the field of values is not considered. Still, the formula is not incorrect because it states that $\br_{i+1} \leq \br_i$.

If $0$ is in the $\bW$-field of values of $\bH\bA$, it can occur that $\gamma_i = \langle \bq_i, \br_i \rangle_\bW =  \langle \bA \bH \br_i -  \sum\limits_{j=0}^{i-1} \Phi_{ij}/\delta_j  \bq_j \rangle_\bW = \langle \bA \bH \br_i, \br_i \rangle_\bW = 0$. Then the residual does not get updated ($\br_{i+1} = \br_i$) and the next search direction is $\bA \bH \br_i$ orthogonalized against all previous ones, including itself. In other words $\bp_{i+1} = \bq_{i+1} = \mathbf 0$. The algorithm has broken down before zero-ing the residual. This is a particularity of the way the search directions are computed in WP-GCR. These unlucky breakdowns do not occur in WP-GMRES if $\bA$ is non-singular (a sufficient but not necessary condition). If an unlucky breakdown occurs in GCR, the algorithm can be restarted by computing the next few search directions as in the weighted and preconditioned Orthodir algorithm.  

A very simple way of understanding that WP-GCR can breakdown is to consider taking $\br_0$ such that $\langle \br_0, \bA \bH \br_0 \rangle_\bW = 0$.  
\end{remark}

\begin{remark}
Minimizing $|\langle \bH \bA \bx, \bx \rangle_\bW| / \langle \bH \bA \bx,\bH \bA  \bx \rangle_\bW $ is equivalent to minimizing\\ $|\langle \by, (\bH \bA)^{-1} \by \rangle_\bW| / \langle \by, \by \rangle_\bW $ by the change of variables $\by = \bH \bA \bx$. This is a way of recovering that the WP-GMRES residual is bounded with respect to the product of the distances to zero of the $\bW$-fields of value of $\bA \bH$ and of $(\bA \bH)^{-1}${ as in \cite{zbMATH01096035}}. 
\end{remark}

\subsection{Restarted and Truncated versions}
\label{sec:restart}

Within WP-GCR as well as WP-GMRES, the new search directions are orthogonalized against all previous ones. The cost of this procedure in terms of computation and storage can become prohibitive if the algorithm takes many iterations to converge. Well established variants of the algorithms have been proposed as early as in \cite{zbMATH03831185,vinsome1976orthomin} where either the orthogonalization is truncated (weighted and preconditioned Orthomin{($k$)}, or WP-Orthomin{($k$)}) or the algorithm is restarted every $k$ iterations (WP-GMRES{($k$)}, WP-GCR{($k$)}). If $k=0$, no orthogonalization is performed at all and the algorithm is called the {weighted and preconditioned Minimal Residual iteration (WP-MR)}. A fact that is not so frequently known is that the convergence result given in Theorem~\ref{th:RPconv} holds for all restarted and truncated versions of the algorithms. This was already fully understood by \cite{zbMATH03831185}. 

\begin{theorem}[Convergence of truncated and restarted versions]
\label{th:RPconvtrunc}

The result in Theorem~\ref{th:RPconv} holds for WP-GCR, WP-GMRES as well as all their truncated and restarted versions including WP-MR.
\end{theorem}
\begin{proof}
{The numbering convention from \cite{zbMATH03831185} is followed. 
\\
\textit{Restarted algorithms: WP-GCR($k$) or WP-GMRES($k$)} During the first cycle (\textit{i.e.}, up to the computation of $\br_{k+1}$ included), Theorem~\ref{th:RPconv} applies. The second cycle consists in $k+1$ iterations of GMRES with initial guess $\bx_{k+1}$. For $k+2 \leq i \leq 2k+2$, the result in Theorem~\ref{th:RPconv} holds since 
\[
 \frac{ \|\br_{k+1} \|_\bW}{\|\br_{0} \|_\bW} \times \frac{ \|\br_{i} \|_\bW}{\|\br_{k+1} \|_\bW} \leq \left( 1 -  \inf\limits_{\by \neq \mathbf 0}  \frac{|\langle {\bA \bH \by}, \by \rangle_\bW|^2}{ \|\bA \bH \by \|_\bW^2  \|\by\|_\bW^2} \right)^{\frac{k+1}{2} + \frac{i-(k+1)}{2}}.  
\]  
and this argument generalizes to any number of restarts. 
\\
\textit{WP-MR and WP-Orthomin($k$)} 
 These algorithms are defined by replacing the formulae for the search directions in Algorithm~\ref{alg:Worthoright} by $\bp_{i+1} = \bz_{i+1} $ for WP-MR and $ \bp_{i+1} = \bz_{i+1}   - \sum\limits_{\max (0, i-k+1)}^{i}\beta_{i,j} \bp_j  $ for WP-Orthomin($k$). The formulae for $\bq_i$ are also modified so that $\bq_i  = \bA \bp_i$.   
Equation \eqref{eq:RPstep0} in the proof of Theorem~\ref{th:RPconv} still holds because the update formula is unchanged. The rest of the proof is direct for WP-MR (by setting $\bq_i = \bA \bH \br_i$) or easily adapted for Orthomin($k$) (by truncating the sums).  
}
\end{proof}

Equations~\eqref{eq:RPstep1} and \eqref{eq:diffwithMR} explain how orthogonalization helps: decreasing $\|\bq_i\|_\bW$ leads to a decrease in the residual and orthogonalization does just that. Full orthogonalization provides the best choice of $\bq_i$ in $\bA \mathcal K_i$. A very important realization is that the field of value convergence bound, as well as all the convergence bounds in this article do not account for the advantages provided by orthogonalization. For this reason they are expected to be over-pessimistic when applied to WP-GMRES and WP-GCR.
 
Two conclusions can be drawn from this. Either the decrease in residual predicted by the convergence bound is satisfactory, then WP-MR, or another truncated or restarted algorithm can be applied. Or, the decrease guaranteed by the theory is not satisfactory, then full WP-GCR/WP-GMRES can be applied with the hope that the bound is over-pessimistic. In practice this is very likely to be the case as WP-GCR/WP-GMRES often exhibits a superlinear convergence behaviour. The convergence bounds of the form presented in this article allow to check that the algorithm cannot stagnate (or near-stagnate), and the superlinear convergence behaviour should kick in. 

All subsequent convergence proofs follow from manipulating the minimized quantity in Theorems~\ref{th:RPconv} and~\ref{th:RPconvtrunc} so they hold also for truncated and restarted versions of the algorithms.  

\subsection{A parenthesis about left preconditioning}

Left preconditioning, \textit{i.e.} solving $\bH \bA {\bx} = \bH \bb$ can be performed instead of right preconditioning. For completeness, the left preconditioned WP-GCR is presented in Algorithm~\ref{alg:Wortholeft}. 
\begin{algorithm}
\caption{WP-GCR with left preconditioning}
\label{alg:Wortholeft}
\begin{algorithmic}
\REQUIRE $\bx_0 \in \mathbb R^n$
\STATE $\br_0= \bb - \bA \bx_{0}$ 
\STATE $\bz_0 = \bH \br_0$ 
\STATE $\bp_0 = \bz_0$
\STATE $\by_0 = \bH \bA \bp_{0}$
\FOR{$i = 0,\,1,\, \dots,\;$convergence}
  \STATE$ \delta_i = \langle \by_i, \by_{i}\rangle_\bW $; \quad $\gamma_i = \langle {\by_{i}}, \bz_{i} \rangle_\bW$; \quad $\alpha_i =  \gamma_i/\delta_i $
  \STATE$ \bx_{i+1} = \bx_{i}+ \alpha_i \bp_i$ 
  \STATE $\bz_{i+1} = \bz_i- \alpha_i \by_i$ 
  \FOR{$j = 0,\,\dots,\, i$}
\STATE  $\Phi_{i,j} = \langle \by_j, \bH \bA \bz_{i+1} \rangle_\bW$; \quad $\beta_{i,j} = \Phi_{i,j}/ \delta_j$
\ENDFOR
\STATE  $\bp_{i+1} = \bz_{i+1} - \sum\limits_{j=0}^{i}\beta_{i,j} \bp_j  $
\STATE  $\by_{i+1} = \bH \bA \bz_{i+1} - \sum\limits_{j=0}^{i} \beta_{i,j}\by_j  $
\ENDFOR
\RETURN{$\bx_{i+1}$}
\end{algorithmic}
\end{algorithm}
{
The  $i$-th residual of Algorithm~\ref{alg:Wortholeft} satisfies
\[
\| \bz_i \|_{\bW} = \|\bH \br_i \|_{\bW} = \operatorname{min}\left\{ \| \bH( \mathbf b - \bA \bx)  \|_{\bW} ; \, {\bx\in\bx_{0} + \operatorname{span} \{\bp_0, \dots, \bp_{i-1}\}}\right\} . 
\]
Moreover, if $0 \not \in  W_\bW(\bH\bA)$, then  $ \operatorname{span} \{\bp_0, \dots, \bp_{i-1}\} = \mathcal K_i$ (the Krylov subspace) so again WP-GCR returns the same approximate solutions as WP-GMRES. 
}
Following the same steps as in Section~\ref{subs:WPGCRright}, it can be proved that 
\[
\frac{\|\bz_{i+1}\|_\bW^2}{\| \bz_i\|_\bW^2} = 1 - \frac{\langle \by_i , \bz_i \rangle_\bW^2}{\langle \by_i, \by_i \rangle_\bW \langle \bz_i, \bz_i \rangle_\bW} \leq  1 - \frac{\langle \bH \bA \bz_i , \bz_i \rangle_\bW^2}{\langle \bH \bA \bz_i, \bH \bA \bz_i \rangle_\bW \langle \bz_i, \bz_i \rangle_\bW} . 
\]
The following observations follow:
\begin{itemize}
\item Left preconditioning produces the same Krylov subspace as right preconditioning.
\item Left preconditioning modifies the residual that is considered in the minimization property as well as the Krylov subspace. 
\item Left preconditioning in the $\bH^{-1}$-inner product and right-preconditioning in the $\bH$-inner product are equivalent (as suggested in \cite{zbMATH01953444}[Problem 9.13]). 
\end{itemize}

\section{Hermitian positive definite preconditioning for positive definite $\bA$}
The convergence study now focuses on some not completely general cases. 
\subsection{Hermitian positive definite preconditioning} 
\label{sec:SymPrec}
Two strong assumptions are made:
\begin{enumerate}
\item the preconditioner $\bH$ is hpd,
\item the inner product is induced by the preconditioner: $\bW = \bH$. 
\end{enumerate}

WP-GCR with right hpd preconditioning and $\bW = \bH$ takes the form of Algorithm~\ref{alg:orthosymprecW}. The name WHP-GCR is adopted where HP stands for Hermitian Preconditioning. The operations have been reorganized so that no additional application of $\bH$ is required compared to unweighted GCR (except in the initialization). The extra cost is the storage of the vectors $\by_j = \bH \bq_j$. Two alternate versions are presented in Algorithms~\ref{alg:altorthosymprecW} and~\ref{alg:altbisorthosymprecW} (of the appendix) that do not require more storage than unweighted GCR. Although no details are given here, the same cost saving measures can be taken in a GMRES algorithm.  Note also that if the preconditioner $\bH$ is very cheap to apply, applying it twice per iteration may be entirely feasible. In this case, it suffices to run a right preconditioned GCR or GMRES code with the inner product changed to $\langle \cdot, \cdot \rangle_\bH$. 

In application of Theorem~\ref{th:WPGCRintro}, WHP-GCR (Algorithm~\ref{alg:orthosymprecW}) is characterized as a Krylov subspace method by the following properties. {First, the vectors $\bq_i =  \bA\bp_i$ are pairwise orthogonal in the $\bH$-inner product.} Second, the residuals satisfy the minimization property: 
$\| \br_i \|_\bH = \operatorname{min}\left\{ \| \mathbf b - \bA \bx \|_{ \bH } ; \, {\bx\in\bx_{0}  + \mathcal K_i} \right\} \text{ if $0 \not \in W_\bH(\bA\bH) $}$.

WP-GCR with right preconditioning by $\bH$ and weighting by $\bW = \bH$ is equivalent to WP-GCR with left preconditioning by $\bH$ and weighting by $\bW = \bH^{-1}$. For this reason the distinction between left and right preconditioning is no longer made. In \cite{zbMATH05626642}, a similar equivalence is observed for an inner product that arises from a symmetric part of the problem matrix $\bA$. 

\begin{algorithm}
\caption{WHP-GCR (\textit{i.e}, WP-GCR with hpd $\bH$ and $\bW = \bH$)} 
\label{alg:orthosymprecW}
\begin{algorithmic}
\REQUIRE $\bx_0 \in \mathbb R^n$
\STATE $\br_0= \bb - \bA \bx_{0}$ 
\STATE $\bz_0 = \bH \br_0$ 
\STATE $\bp_0 = \bz_0$
\STATE $\bq_0 = \bA \bp_0$
\STATE $\by_0 = \bH \bq_{0}$
\FOR{$i = 0,\,1,\, \dots,\;$convergence}
\STATE  $\delta_i = \langle \by_i, \bq_{i}\rangle $; \quad $\gamma_i = \langle {\bq_{i}}, \bz_{i} \rangle$; \quad $\alpha_i =  \gamma_i/\delta_i $
\STATE   $\bx_{i+1} = \bx_{i}+ \alpha_i \bp_i$ 
\STATE   $\br_{i+1} = \br_{i}- \alpha_i \bq_i$ 
\STATE   $\bz_{i+1} = \bz_i- \alpha_i \by_i$ 
\STATE   $\bp_{i+1} = \bz_{i+1}$
\STATE   $\bq_{i+1} = \bA \bz_{i+1}$ 
  \FOR{$j = 0,\,\dots,\, i$}
\STATE  $\Phi_{i,j} = \langle \by_j,   \bq_{i+1} \rangle$; \quad $\beta_{i,j} = \Phi_{i,j}/ \delta_j^{-1} $
\STATE  $\bp_{i+1} -= \sum\limits_{j=0}^{i}\beta_{i,j} \bp_j  $
\STATE  $\bq_{i+1} -= \sum\limits_{j=0}^{i}\beta_{i,j} \bq_j $ 
  \ENDFOR
\STATE $\by_{i+1} = \bH \bq_{i+1}$ 
\ENDFOR
\RETURN{$\bx_{i+1}$}
\end{algorithmic}
\end{algorithm}

\begin{theorem}[Convergence of WHP-GCR]
\label{th:HPconv}
Assume that the operator $\bA \in \mathbb K^{n\times n}$ is non-singular. Assume also that the preconditioner $\bH \in \mathbb K^{n\times n}$ is hpd and that right preconditioned GCR is applied in the inner product induced by $\bH$.
The $i$-th iterate of Algorithm~\ref{alg:orthosymprecW} satisfies
\begin{equation}
\label{eq:HPsuff}
\frac{\| \br_{i+1} \|_\bH}{\| \br_{0} \|_\bH}\leq \left[ 1- \inf\limits_{\by\neq 0} \frac{|\langle \bM( \bA^{-1}) \by, \by \rangle|}{ \langle \bH \by, \by \rangle} \times \inf\limits_{\by\neq 0} \frac{|\langle \bM(\bA) \by, \by \rangle|}{ \langle \bH^{-1} \by, \by \rangle}\\\right]^{i/2}. 
\end{equation}
The same result holds for all truncated and restarted versions of WHP-GCR and WHP-GMRES, including WHP-MR (\textit{i.e.}, WP-MR with $\bW = \bH$ hpd). 
\end{theorem}
\begin{proof}
Applying Theorem~\ref{th:RPconv} to WHP-GCR, the quantity that must be bounded can be rewritten as 
\begin{align*}
\inf\limits_{\by\neq 0} \frac{|\langle {\bA \bH \by}, \by \rangle_\bH|^2}{ \|\bA \bH \by \|_\bH^2  \|\by\|_\bH^2} &=  \inf\limits_{\by\neq 0} \frac{|\langle \bH \bA \bH \by, \by \rangle|^2}{ \langle \bH \bA \bH \by, \bA \bH \by \rangle \langle \bH \by, \by \rangle} \\ 
 &=  \inf\limits_{\by\neq 0} \frac{|\langle \bA \by, \by \rangle|^2}{ \langle \bH \bA \by, \bA \by \rangle \langle \bH^{-1} \by, \by \rangle}\hfill \text{ (by $\by \leftarrow \bH\by$)} \\ 
 &\geq \inf\limits_{\by\neq 0} \frac{|\langle \bA \by, \by \rangle|}{ \langle \bH \bA \by, \bA \by \rangle}\times  \inf\limits_{\by\neq 0} \frac{|\langle \bA \by, \by \rangle|}{ \langle \bH^{-1} \by, \by \rangle}\\
 &= \inf\limits_{\by\neq 0} \frac{|\langle \bA^{-1} \by, \by \rangle|}{ \langle \bH \by, \by \rangle} \times \inf\limits_{\by\neq 0} \frac{|\langle \bA \by, \by \rangle|}{ \langle \bH^{-1} \by, \by \rangle}\hfill \text{ (by $\by \leftarrow \bA\by$)} \\ 
 &{\geq \inf\limits_{\by\neq 0} \frac{|\langle \bM( \bA^{-1}) \by, \by \rangle|}{ \langle \bH \by, \by \rangle} \times \inf\limits_{\by\neq 0} \frac{|\langle \bM(\bA) \by, \by \rangle|}{ \langle \bH^{-1} \by, \by \rangle}
}
\end{align*}
{
where, $\bM(\bA)$ and  $\bM(\bA^{-1})$ are the Hermitian parts of $\bA$ and $\bA^{-1}$ (as introduced in \eqref{eq:hsh})  and $|\langle \bB \by, \by \rangle| \geq| \operatorname{Re}(\langle \bB \by, \by \rangle) | =  |\langle \bM(\bB) \by, \by \rangle|  $ for any matrix $\bB \in \mathbb K^{n\times n}$. 
}
\end{proof}

\subsection{Positive definite $\bA$}
\label{sec:PD}

Assuming that the problem matrix $\bA$ is positive definite, \textit{i.e.}, that $\bM(\bA)$ is hpd, the calculations from the previous paragraph are resumed: 
\begin{align*}
\inf\limits_{\by\neq 0} \frac{\langle {\bA \bH \by}, \by \rangle_\bW^2}{ \|\bA \bH \by \|_\bW^2  \|\by\|_\bW^2} 
 &\geq \inf\limits_{\by\neq 0} \frac{|\langle \bM( \bA^{-1}) \by, \by \rangle|}{ |\langle \bM(\bA)^{-1} \by, \by \rangle |} &\times \inf\limits_{\by\neq 0} \frac{|\langle \bM(\bA)^{-1}  \by, \by \rangle|}{ \langle \bH \by, \by \rangle}\\
& &  \times \inf\limits_{\by\neq 0} \frac{|\langle \bM(\bA) \by, \by \rangle|}{ \langle \bH^{-1} \by, \by \rangle}\\
 &\geq \inf\limits_{\by\neq 0} \frac{\langle \bM( \bA^{-1}) \by, \by \rangle}{ \langle \bM(\bA)^{-1} \by, \by \rangle } &\times \inf\limits_{\by\neq 0} \frac{\langle \bM(\bA)^{-1}  \by, \by \rangle}{ \langle \bH \by, \by \rangle}\\
& & \times \inf\limits_{\by\neq 0} \frac{\langle \bM(\bA) \by, \by \rangle}{ \langle \bH^{-1} \by, \by \rangle},
\end{align*}
{where division by $\langle \bM(\bA)^{-1} \by, \by \rangle$ is not by zero} and removing the absolute values does not change the result. Indeed, by  \cite{johnson1972matrices}[Property (1.6) on page 10],  $\bM(\bA)$ being positive definite implies that $\bA^{-1}$ is well defined and that $\bM(\bA^{-1})$ is also positive definite.

Let $\lambda_{\min} (\bH\bM(\bA))$ and $\lambda_{\max} (\bH\bM(\bA))$ denote the smallest and largest eigenvalues of $\bH \bM(\bA)$. The eigenvalues of $\bH\bM(\bA)$, are also the eigenvalues of the generalized eigenvalue problems $ \bM(\bA)\bH\bM(\bA) \by = \lambda \bM(\bA) \by$ and $ \bH\bM(\bA)\bH \by = \lambda \bH \by$. By the Rayleigh-Ritz characterization of eigenvalues, an interpretation of the last two terms in the product follows:
\[
\inf\limits_{\by\neq 0} \frac{\langle \bM(\bA) \by, \by \rangle}{ \langle \bH^{-1} \by, \by \rangle} = \inf\limits_{\by\neq 0} \frac{\langle \bH \bM(\bA) \bH \by, \by \rangle}{ \langle \bH \by, \by \rangle} = \lambda_{\min} (\bH\bM(\bA)).
\]
and
\[
\inf\limits_{\by\neq 0} \frac{\langle \bM(\bA)^{-1}  \by, \by \rangle}{ \langle \bH \by, \by \rangle} =  \left(\sup\limits_{\by\neq 0} \frac{ \langle  \bM(\bA)\bH\bM(\bA) \by, \by \rangle}{ \langle  \bM(\bA) \by, \by \rangle}\right)^{-1} = \frac{1}{\lambda_{\max} (\bH\bM(\bA))}, 
\]

For the remaining term, recalling from \eqref{eq:hsh} that $\bN(\bA)$ denotes the skew Hermitian part of $\bA$, a very useful result is the following.
\begin{lemma}[Corollary 3 in \cite{zbMATH03489299}, see also \cite{zbMATH03389168,johnson1972matrices}]
Let $\bA \in \mathbb C^n$ be positive definite and $c \in \mathbb R$. The matrix $c \bM(\bA ^{-1}) - \bM(\bA)^{-1}$  is positive definite
 if and only if $c > 1 + \rho(\bM(\bA)^{-1} \bN(\bA)) ^2$, with 
\[
\text{$\rho(\bM(\bA)^{-1} \bN(\bA))  = \max \{|t_j|; \, \pm i t_j$ are the eigenvalues of $\bM(\bA)^{-1} \bN(\bA) $\}.} 
\]
\end{lemma}

An equivalent reformulation of the lemma is that 
\begin{equation}
\label{eq:thanks_johnson}
\inf\limits_{\by\neq 0} \frac{\langle \bM( \bA^{-1}) \by, \by \rangle}{ \langle \bM(\bA)^{-1} \by, \by \rangle } = [1 + \rho(\bM(\bA)^{-1} \bN(\bA) )^2]^{-1}.
\end{equation}
By definition, $\rho(\bM(\bA)^{-1} \bN(\bA))$ is the spectral radius of $\bM(\bA)^{-1} \bN(\bA)$, $\textit{i.e.}$, the norm of the eigenvalue of maximal norm. In this particular case, the eigenvalues of $(\bM(\bA)^{-1} \bN(\bA) )$ are conjugate pairs of {purely} imaginary numbers {$\pm i |t_j|$, as well as $0$ if the order of $\bA$ is odd.

Putting everything together, we get:
\begin{theorem}[Convergence of WHP-GCR for positive definite $\bA$]
\label{th:final}
Assume that the operator $\bA \in \mathbb K^{n\times n}$ is positive definite. Assume also that the preconditioner $\bH \in \mathbb K^{n\times n}$ is hpd and that right preconditioned GCR is applied in the inner product induced by $\bH$.
The $i$-th iterate of Algorithm~\ref{alg:orthosymprecW} satisfies
\begin{equation}
\label{eq:final}
\frac{ \|\br_{i} \|_\bH}{\|\br_{0} \|_\bH} \leq \left[ 1 - \frac{\lambda_{\min} (\bH\bM(\bA))}{\lambda_{\max} (\bH\bM(\bA))} \times  \frac{1}{1 + \rho(\bM(\bA)^{-1} \bN(\bA))^2  } \right]^{i/2}. 
\end{equation}
The same result holds for all truncated and restarted versions of WHP-GCR and WHP-GMRES, including WHP-MR. 
\end{theorem}

The estimate in the theorem has split the residual bound into two components:
\begin{itemize}
\item the condition number of the Hermitian part of $\bA$ preconditioned by $\bH$, 
\item a measure of the strength of non-Hermitianness of $\bA$ that is independent of $\bH$.
\end{itemize}

\begin{corollary}
\label{cor:final}
Under the assumptions of Theorem~\ref{th:final}, if $\bH$ is a preconditioner for $\bM(\bA)$ such that the condition number of the preconditioned operator is $\kappa(\bH \bM(\bA))$ then 
\begin{equation}
\frac{ \|\br_{i} \|_\bH}{\|\br_{0} \|_\bH} \leq \left[ 1 -  \frac{\kappa(\bH \bM(\bA))^{-1}}{1 + \rho(\bM(\bA)^{-1} \bN(\bA) )^2 } \right]^{i/2}. 
\end{equation}
The same result holds for all truncated and restarted versions of WHP-GCR and WHP-GMRES, including WHP-MR. 
\end{corollary}

In particular, if $\bH$ is a domain decomposition preconditioner such that $\kappa(\bH \bM(\bA))$ can be bounded independently of the number of subdomains, then the convergence bound above does not depend on the number of subdomains either. In other words, if a scalable domain decomposition method is known for the Hermitian part of $\bA$, the proposed algorithm for the non-Hermitian problem is also scalable. For many problems arising from the discretization of PDEs, $\bM(\bA)$ is derived from a differential operator of lower order than that producing $\bN(\bA)$ so $\rho(\bM(\bA)^{-1} \bN(\bA) )$ is bounded independently of the discretization step (an observation already made in a different context in \cite{zbMATH03619254} and \cite{elman1982iterative}[Section 10]). If $\bH$ can also be chosen so that the preconditioned Hermitian part is well conditioned independently of $h$ then the residual bound is $h$-independent.

Next, it is illustrated that WHP-GCR is efficient for mildly non-Hermitian problems. If the preconditioner is well chosen, the algorithm is scalable and optimal (in the sense that convergence does not depend on the discretization step).

\section{Illustration: Convection-Diffusion-Reaction}
\label{sec:Numerical}

In this section, the problem considered is the convection-diffusion-reaction problem posed in $\Omega = [0,1]^2$. It is a real-valued problem ($\mathbb K = \mathbb R$), so Hermitian means symmetric. The presentation, notation and test case are inspired by \cite{zbMATH07395831,bonazzoli:hal-03882577}. The strong formulation of the problem is: 
\begin{align*}
c_0 u + \operatorname{div}(\mathbf a u) - \operatorname{div} (\nu \nabla u) &= f \text{ in } \Omega,\\ 
u &= 0 \text{ on } \partial \Omega.
\end{align*}

The variational formulation is: 
Find $u \in H^1_0(\Omega)$ such that 
\[
\underbrace{\int_\Omega \left(\left(c_0 + \frac{1}{2} \operatorname{div} \mathbf a \right) uv+ \nu \nabla u \cdot \nabla v \right)}_{\text{symmetric part}}  + \int_\Omega \left(\frac{1}{2} \mathbf a \cdot \nabla u v - \frac{1}{2} \mathbf a \cdot \nabla v u \right)  = \int_\Omega fv, 
\]
for all $ v \in H^1_0(\Omega)$. 

The right hand side and the convection field are chosen as 
\[
f(x,y) =  \operatorname{exp} (-10((x - 0.5)^2 + (y - 0.1)^2)) \text{ and } \mathbf a(x,y) = 2\pi[-(y-0.1), x-0.5]. 
\]
The reaction coefficient $c_0 >0$ and viscosity $\nu >0 $ are chosen to be constant over $\Omega$. {Together with $ \operatorname{div} \mathbf a =0$, positivity of $c_0$ and $\nu$ ensures that $\bA$ is positive-definite.} Varying $c_0$ and $\nu$ {inside $\Omega$} would not cause any additional difficulty since the proposed preconditioner handles heterogeneous $c_0$ and $\nu$ (in the Hermitian part). The problem is discretized by Lagrange $\mathbb P_1$ finite elements {on a regular triangular mesh of characteristic length $h$}. The WHP-GCR algorithm is implemented in FreeFem++ \cite{MR3043640} with the ffddm library \cite{FFD:Tournier:2019}.  Except in one case (where it is specified otherwise), all iteration counts for WHP-GCR correspond to the number of iterations needed to reach $\|\br_i\|_\bH < 10^{-6} \|\bb \|_\bH$ starting from a zero initial guess. {The Dirichlet boundary condition has been enforced by penalization. The solution computed by a direct solve on a $501\times 501$ degree of freedom (dof) mesh has been plotted in Figure~\ref{fig:16sub} (left) for different values of $c_0 = \nu = 0.1$.} 
\begin{figure}
\begin{center}
\includegraphics[width=0.18 \textwidth]{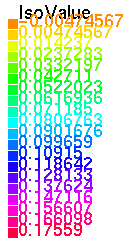}
\includegraphics[width=0.35 \textwidth]{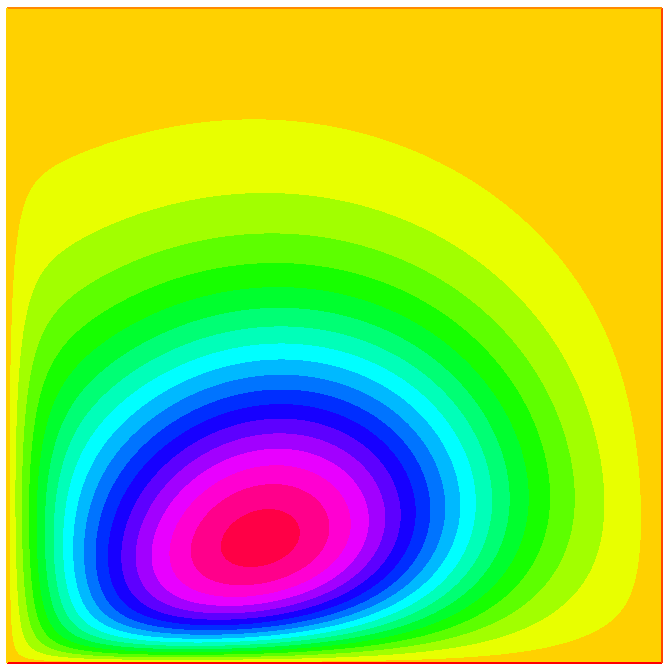}
\includegraphics[width = 0.35 \textwidth]{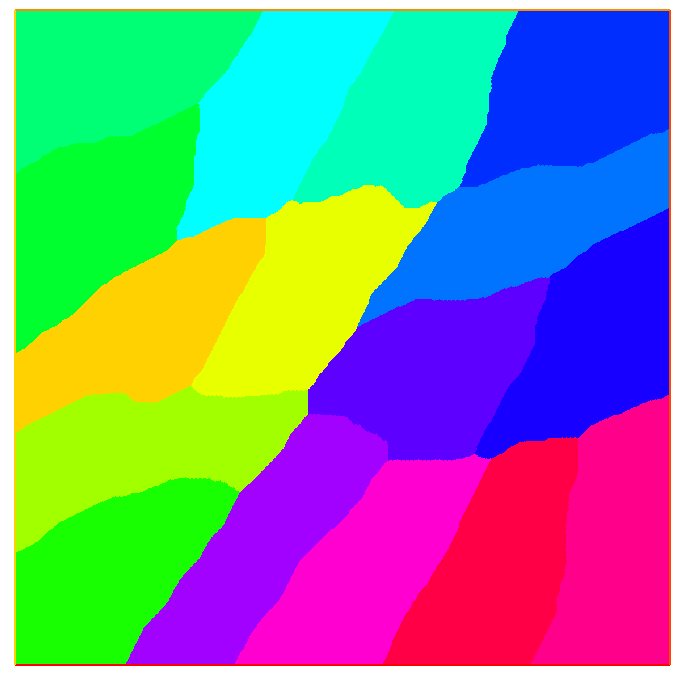}
\end{center}
\caption{Left: Solution for $c_0 = \nu = 0.1$. Right: Partition into $N=16$ subdomains computed by Metis.}
\label{fig:16sub}
\end{figure}

Let $\bM(\bA)$ be the matrix corresponding to the discretization of the symmetric part of the problem. $\bM(\bA)$ is preconditioned by the Additive Schwarz domain decomposition method with the GenEO coarse space \cite{2011SpillaneCR,spillane2013abstract}. The partition of $\Omega$ into $N$ subdomains $\Omega\s$ is computed automatically by Metis. One layer of overlap is added to each $\Omega\s$. Letting ${\bR\s}^\top$ ($s=1,\dots,N$) denote the prolongation by zero of local finite element functions (in $\Omega\s$) to the whole of $\Omega$, the preconditioner can be written as: 
\[
\bH = \bPi \sum\limits_{s=1}^N {\bR\s}^\top \underbrace{(\bR\s \bM(\bA) {\bR\s}^\top)^{-1}}_{\text{local solves}} \bR\s \bPi^\top + {\bR\0}^\top \underbrace{(\bR\0 \bM(\bA) {\bR\0}^\top)^{-1}}_{\text{coarse solve}} \bR\0, 
\]
where $\bPi = \matid -  {\bR\0}^\top (\bR\0 \bM(\bA) {\bR\0}^\top)^{-1} \bR\0 \bM(\bA)$ is the coarse projector (also known as a deflation operator) and the vectors in $\bR_0^\top$ span the coarse space (or deflation space). The particularity of GenEO is that the coarse vectors are constructed by solving the low frequency eigenmodes for a generalized eigenvalue problem in each subdomain. The user provides a threshold $\tau > 0$, \textit{e.g.}, $\tau = 0.15$. The corresponding ffddm options are 

\verb+-ffddm_schwarz_method asm+ 

\verb+-ffddm_geneo_threshold 0.15+ 

\verb+-ffddm_schwarz_coarse_correction BNN+. 

{
The condition number of the resulting preconditioned operator is bounded by
\[
\kappa(\bH \bM(\bA)) \leq k_0 \left(1 + \frac{k_0}{\tau}\right), 
\] 
where $k_0$ denotes the maximal number of subdomains that each mesh element belongs to \cite{spill2014}[Theorem 4.40]. This constant does not depend on the total number $N$ of subdomains. In all examples $\tau$ has been set to $0.15$. As an illustration, the partition into $N = 16$ subdomains provided by Metis is shown in Figure~\ref{fig:16sub} (right). For this case, it holds that $k_0 = 3$. Consequently, the condition number of the preconditioned symmetric part of the problem is bounded from above by $\kappa(\bH \bM(\bA)) \leq 3(1+3 / 0.15) = 63$. 
Injecting this into the bound from Theorem~\ref{th:final}, it is obtained that the residuals produced by WHP-GCR or WHP-GMRES satisfy 
\[
\frac{ \|\br_{i} \|_\bH}{\|\br_{0} \|_\bH} \leq \left[ 1 -  \frac{0.016}{1 + \rho(\bM(\bA)^{-1} \bN(\bA))^2  } \right]^{i/2}. 
\]
If, for example $\rho(\bM(\bA)^{-1} \bN(\bA)) \leq 1$, the bound gives 
\[
\frac{ \|\br_{i} \|_\bH}{\|\br_{0} \|_\bH} \leq {\sqrt{\left[ 1 -  \frac{1}{126} \right]}}\,^{i} = 0.996^i; \, \frac{ \|\br_{500} \|_\bH}{\|\br_{0} \|_\bH} \leq 0.14 \text{ and } \frac{ \|\br_{3468} \|_\bH}{\|\br_{0} \|_\bH} < 1.0\cdot10^{-6} .  
\]
The residual is guaranteed to decrease only by $0.4 \%$ at each iteration. As was previously explained, the bound is pessimistic for fully orthogonalized WP-GCR and WP-GMRES and we expect in practice to observe much faster decrease in residual.
}

\paragraph{Value of $\rho(\bM(\bA)^{-1} \bN(\bA))$} It remains to consider the value of $\rho(\bM(\bA)^{-1} \bN(\bA))$. In the proof of \cite{zbMATH07395831}[Lemma 4.5], by Cauchy Schwarz and some arithmetic identities, it is proved that
\[
\langle \by, \bN(\bA) \bx \rangle \leq \alpha \|\bx \|_{\bM(\bA)} \| \by\|_{\bM(\bA)}, \quad \forall \bx, \by \in \mathbb R^n.
\]
with $\alpha = \frac{1}{2} \frac{\|\mathbf{a}\|_{L^\infty(\Omega)}}{\sqrt{\inf (\nu) \inf (c_0 + \frac{1}{2} \operatorname{div} (\mathbf a))}}$.  Consequently, for $\bx \neq \mathbf 0$, it holds that 
\begin{align*}
\| \bM(\bA)^{-1} \bN(\bA)\bx \|_{\bM(\bA)}^2 &\leq \alpha \langle \bx, \bM(\bA) \bx \rangle^{1/2} \langle \bN(\bA) \bx, \bM(\bA)^{-1} \bN(\bA) \bx \rangle^{1/2}\\ 
& = \alpha \| \bx\|_{\bM(\bA)} \| \bM(\bA)^{-1} \bN(\bA) \bx\|_{\bM(\bA)} . 
\end{align*} 
The spectral radius of a matrix is bounded from above by any of its natural norms, and in particular by the norm induced by $\bM(\bA)$, from which it immediately follows that 
\begin{equation}
\label{eq:boundrho}
\rho(\bM(\bA)^{-1} \bN(\bA)) \leq \|\bM(\bA)^{-1} \bN(\bA) \|_{\bM(\bA)} \leq \frac{1}{2} \frac{\|\mathbf{a}\|_{L^\infty(\Omega)}}{\sqrt{\inf (\nu) \inf (c_0 + \frac{1}{2} \operatorname{div} (\mathbf a))}}. 
\end{equation}
Bound \eqref{eq:boundrho} for $\rho(\bM(\bA)^{-1} \bN(\bA))$ does not depend on the discretization step $h$ so neither does the overall convergence bound. 

In Table~\ref{tab:rho}, the actual value of $\rho(\bM(\bA)^{-1} \bN(\bA))$ computed by Octave's \textit{eigs} is given in the case where $c_0 = \nu = 1$. The discretization step varies between $h = 1/10$ and $h=1/200$ and $\rho(\bM(\bA)^{-1} \bN(\bA))$ varies only between $0.31$ and $0.34$. 
In comparison, bound \eqref{eq:boundrho} gives $\rho(\bM(\bA)^{-1} \bN(\bA)) \leq 3.23$. The bound is approximately $10$ times larger than the computed value of the spectral radius.  

\begin{table}
\centering
\begin{tabular}{|c|c|c|c|c|}
\hline
Discretization step $h$ & {$1/500$} & $1/200$ & $1/30$ & $1/10$ \\
\hline
$\rho(\bM(\bA)^{-1} \bN(\bA))$ & 0.3391 &  0.3389 & 0.3380 & 0.3136 \\
\hline
\end{tabular}
\caption{Computation of $\rho(\bM(\bA)^{-1} \bN(\bA))$ when the discretization step $h$ varies. Case $\nu = c_0 = 1$. } 
\label{tab:rho}
\end{table}

\paragraph{Scalability (Table~\ref{tab:scal})}
Since Additive Schwarz with the GenEO coarse space is scalable, it has been proved that the overall algorithm is scalable. This is checked by solving the same problem for different partitions into subdomains (all computed by Metis). Two discretizations are considered: $h = 1/200$ and $h = 1/500$. For this test $\nu = c_0 = 1$. It is observed that the method is indeed scalable: the iteration counts reported in Table~\ref{tab:scal} do not depend on the number of subdomains. A dependency on $h$ is observed and this is studied next.

\begin{table}
\centering
\begin{tabular}{|c|c|c|c|c|}
\hline
Number of subdomains & 4 & 8 & 16 & 32 \\
\hline
Iteration count for $h= 1/200$ & 19 & 20 & 20 & 20\\ 
\hline
Iteration count for $h= 1/500$ & 18 & 18 & 19 & 20 \\ 
\hline
\end{tabular}
\caption{Scalability. In each line, the same problem is solved for an increasing number of subdomains. The iteration count remains constant. Case $\nu = c_0 = 1$.}
\label{tab:scal}
\end{table}

{
\paragraph{Dependency on $h$ (Table~\ref{tab:h})}
The influence of the discretization step $h$ on the iteration count is studied in Table~\ref{tab:h}. The partition is set to $N = 8$ subdomains. For, three different values of $c_0 = \nu $, the mesh size $h$ varies between $1/100$ and $1/2000$. It is observed that, as predicted by the theory, the number of iterations remains almost constant when $h$ varies.
}

\begin{table}
\centering
\begin{tabular}{|c|c|c|c|c|c|c|}
\hline
&$1/h$ & 2000 & 1000 & 500 & 200 & 100   \\
\hline
&$\#$ dofs ($n$) &4 004 001 & 1 002 001 & 251 001 &  40 401 & 10 201 \\ 
\hline
$c_0 = \nu = 10$ & it. count & 16 & 16 & 17 & 17 & 20 \\
$c_0 = \nu = 1$  &it. count  & 17 & 18 & 19 & 20 & 21 \\ 
$c_0 = \nu = 0.1 $&it. count & 39 &40  &42  &43  &41  \\ 
\hline
\end{tabular}
\caption{Dependency on mesh size $h$. Each line corresponds to a value of $(\nu, c_0)$.}
\label{tab:h}
\end{table}

\paragraph{Dependency on strength of non-symmetry (Table~\ref{tab:nonsym})}
For this test, $h = 1/500$ and $N = 8$. The value of $\nu$ and $c_0$ are varied and the corresponding iteration counts are reported in Table~\ref{tab:nonsym}. The problem converges very fast when the symmetric part dominates and not so fast otherwise. This is expected by design of the algorithm.   

{
\begin{table}
\centering
\begin{tabular}{|c|c|c|c|c|c|c|c|}
\hline
$c_0 = \nu$ & 0.001 & 0.01 & 0.1 &  1 &  10 & \text{only symmetric part} \\
\hline
Iteration count & $>$ 500 ( $1.1 \cdot 10^{-4}$) & 161 & 42 & 19 & 17 & 17 \\
\hline
\end{tabular}
\caption{Influence of the relative importance of the symmetric term and the skew-symmetric term. Case $N = 8$ subdomains and $h = 1/500$. The value in parenthesis corresponds to the relative residual $\| \br_{500}\|_\bH / \| \rhs\|_\bH$ when the algorithm stopped after $500$ iterations.} 
\label{tab:nonsym}
\end{table}
}

\paragraph{GMRES in the Euclidean inner product (Table~\ref{tab:norm})} It has already been observed that changing the inner product in GMRES does not influence its convergence very much \cite{zbMATH05626642,zbMATH07726053}. A scalability test is run for $h = 1/500$ in two settings $\nu = c_0 = 1$ and   $\nu = c_0 = 10$. The number of subdomains varies between $4$ and $32$. Two algorithms are applied with the same preconditioner as previously: right preconditioned GMRES and WHP-GCR. The difference is that GMRES works in the Euclidean inner product while WHP-GCR (which produces the same iterates as WHP-GMRES) works in the $\bH$-inner product. For both, the stopping criterion has been set to $\|\br_i\| < 10^{-6} \| \bb \| $ where the norm is the Euclidean norm. It is remarkable that the iteration counts reported in Table~\ref{tab:norm} are almost identical. This means that, in practice, using the hpd preconditioner and the Euclidean norm will most likely give results that are in agreement with the developed theory. The advantage is to save the effort of implementing WHP-GCR if GMRES is already available. The extra cost of running WHP-GCR is small compared to GCR (or GMRES) since it is only the cost of storing one (or two) extra vectors per iteration. 

\begin{table}
\centering
{\bf Case $\nu = c_0 = 1$}

\begin{tabular}{|c|c|c|c|c|}
\hline
Number of Subdomains & 4 & 8 & 16 & 32 \\
\hline
GMRES & 24 & 25 & 26 & 26 \\ 
\hline
WHP-GCR (Euclidean stopping criterion) & 25 & 26 & 26 & 27 \\ 
\hline
\end{tabular}

\vspace{0.5cm}
{\bf Case $\nu = c_0 = 0.1$}

\begin{tabular}{|c|c|c|c|c|}
\hline
Number of Subdomains & 4 & 8 & 16 & 32 \\
\hline
GMRES & 52 & 52 & 53 & 52 \\ 
\hline
WHP-GCR (Euclidean stopping criterion) & 53 & 53 & 55 & 53 \\
\hline
\end{tabular}
\caption{Influence of the inner product: GMRES in Euclidean norm compared to WHP-GCR. The stopping criterion is in Euclidean norm for both algorithms. Case $h = 1/500$. The number of subdomains $N$ varies and two cases are considered : $\nu = c_0 = 1$ and $\nu = c_0 = 0.1$.} 
\label{tab:norm}
\end{table}

{
\paragraph{Comparison with a non-symmetric preconditioner} We finally compare symmetric and non-symmetric preconditioning. Since WHP-GCR cannot be applied with a non-symmetric preconditioner, the GMRES solver is used for this comparison. The stopping criterion is again set to $\|\br_i\| < 10^{-6} \| \bb \| $. As a non-symmetric preconditioner the one-level additive Schwarz method has been selected. To make the comparison fair, we also precondition by the symmetric one-level additive Schwarz preconditioner corresponding to $\bM(\bA)$. Finally, we include the two-level symmetric preconditioner which has been applied in all previous tests. The results are presented in Table~\ref{tab:symnonsym}. As soon as the problem becomes significantly non-symmetric ($c_0= \nu \geq 0.1 $) there is a clear advantage for the non-symmetric preconditioner compared to the one-level symmetric preconditioner. For $c_0= \nu = 0.01$, the one-level non-symmetric preconditioner converges much faster even than even the two-level symmetric preconditioner. This is not very surprising since the symmetric preconditioner does not account at all for the convective term. It is to be noted however that this non-symmetric preconditioner will deteriorate when $N$ increases and when $h$ decreases (as briefly illustrated in Figure~\ref{tab:badnonsym}). 
}

\begin{table}
\centering
\begin{tabular}{|c|c|c|c|c|c|c|}
\hline
$c_0 = \nu$ &  0.01 & 0.1 &  1 &  10 & \text{only symmetric part} \\
\hline
Non-sym one-level & 35 & 68 & 81 & 81 & $\times$\\ 
Sym one-level & $>200$ & 105 & 87 & 84 & 81 \\ 
Sym two-level & 191 & 52 & 25 & 23 & 24 \\ 
\hline
\end{tabular}
\caption{GMRES Iteration counts. One-level non-symmetric preconditioner is compared to one-level symmetric preconditioner and two-level symmetric preconditioner.  Case $N = 8$ subdomains and $h = 1/500$.}
\label{tab:symnonsym}
\end{table}

\begin{table}
\centering
\begin{tabular}{|c|c|c|c|c|}
\hline
\multicolumn{5}{|c|}{Non-symmetric preconditioning}\\
\hline
&\multicolumn{2}{|c|}{$8$ subdomains} &\multicolumn{2}{|c|}{$16$ subdomains}\\ 
\hline
 & $h=1/500$ & $h = 1/1000$  & $h=1/500$ & $h = 1/1000$ \\
\hline
$c_0 = \nu=  0.01$  & 35 & 58 & 47 & 67 \\ 
$c_0 = \nu=  0.1$  & 68 & 96 & 82 & 113 \\ 
\hline
\end{tabular}
\caption{GMRES Iteration counts with the non-symmetric one-level Additive Schwarz preconditioner. Convergence deteriorates when number of subdomains increases or mesh size $h$ decreases.}
\label{tab:badnonsym}
\end{table}

\section{Conclusion}
In this article, the convergence of GMRES, GCR and their truncated and restarted versions has been studied. The influence of the preconditioner and the inner product have been made explicit. It has been proposed, even for non-Hermitian problems to apply a hpd preconditioner $\bH$. Then, GMRES or GCR can be applied in the $\bH$ inner product. This is referred to as WHP-GMRES. A new convergence result is proved for cases where $\bA$ is positive definite. The two terms in the convergence bound are the condition number of the Hermitian part of $\bA$ once preconditioned by $\bH$ and the spectral radius of $\bM(\bA)^{-1}\bN(\bA)$. This last term can be seen as a measure of the strength of non-Hermitianness. It does not depend on the choice of the preconditioner. A particular application is the case where $\bH$ is a domain decomposition preconditioner. If the preconditioner applied to the Hermitian part of $\bA$ leads to a scalable method, then WHP-GCR will be scalable too. {For the Convection-Diffusion-Reaction problem, it has also been proved that convergence will not depend on the mesh size $h$ as long as the condition number of the preconditioned symmetric part does not depend on $h$.} Numerical results have confirmed these findings. It remains to improve the algorithm in cases where the problem is strongly non-Hermitian or indefinite.  
\appendix

\section{Implementation of WHP-GCR}
Algorithms~\ref{alg:altorthosymprecW} and ~\ref{alg:altbisorthosymprecW} propose two alternate implementations of WHP-GCR (Algorithm~\ref{alg:orthosymprecW}) that require no more storage that the usual GCR algorithm.

The notation $\tilde\cdot$ has been used to emphasize that vectors with a tilde do not get orthogonalized and saved. In exact arithmetic all three versions produce the same iterates. The Euclidean residual $\br_i$ is not updated in Algorithm~\ref{alg:altorthosymprecW} which may be a drawback. In finite precision, computing $\alpha_i$ from vectors that have not been explicitly orthogonalized could lead to inaccuracy. This is why Algorithm~\ref{alg:orthosymprecW} is emphasized and implemented in the numerical result section. 

\begin{minipage}{0.45\textwidth}
\begin{algorithm}[H]
\caption{Alternate WHP-GCR}
\label{alg:altorthosymprecW}
\begin{algorithmic}
\REQUIRE $\bx_0 \in \mathbb R^n$
\STATE $\bz_0 = \bH (\bb - \bA \bx_{0})$  
\STATE $\bp_0 = \bz_0$
\STATE $\tilde \bq_0 = \bA \bp_0$
\STATE $\by_0 = \bH \tilde \bq_{0}$
\FOR{$i = 0,\,1,\, \dots,\;$convergence}
 \STATE $\delta_i = \langle \by_i, \tilde \bq_{i}\rangle $; \quad $\gamma_i = \langle {\tilde \bq_{i}}, \bz_{i} \rangle$
 \STATE  $\alpha_i =  \gamma_i/\delta_i $
 \STATE $\bx_{i+1} = \bx_{i}+ \alpha_i \bp_i$ 
 \STATE 
 \STATE $\bz_{i+1} = \bz_i- \alpha_i \by_i$ 
 \STATE $\bp_{i+1} = \bz_{i+1}$
 \STATE $\tilde \bq_{i+1} = \bA \bz_{i+1}$ 
 \STATE $\by_{i+1} = \bH \tilde \bq_{i+1}$ 
  \FOR{$j = 0,\,\dots,\, i$}
\STATE  $\Phi_{i,j} = \langle \by_j,  \tilde \bq_{i+1} \rangle$
\STATE  $\beta_{i,j} = \Phi_{i,j}/ \delta_j^{-1} $
\ENDFOR
\STATE  $\bp_{i+1} -= \sum\limits_{j=0}^{i}\beta_{i,j} \bp_j  $
\STATE  $\by_{i+1} -= \sum\limits_{j=0}^{i}\beta_{i,j} \by_j  $
\ENDFOR  
\RETURN{$\bx_{i+1}$}
\end{algorithmic}
\end{algorithm}
\end{minipage}
\begin{minipage}{0.45\textwidth}
\begin{algorithm}[H]
\caption{Alternate WHP-GCR}
\label{alg:altbisorthosymprecW}
\begin{algorithmic}
\REQUIRE $\bx_0 \in \mathbb R^n$
\STATE $\bz_0 = \bH (\bb - \bA \bx_{0})$  
\STATE $\bp_0 = \bz_0$
\STATE $\bq_0 = \bA \bp_0$
\STATE $\tilde \by_0 = \bH \bq_{0}$
\FOR{$i = 0,\,1,\, \dots,\;$convergence}
\STATE  $\delta_i = \langle\tilde  \by_i, \bq_{i}\rangle $; \quad $\gamma_i = \langle { \bq_{i}}, \bz_{i} \rangle$
\STATE $\alpha_i =  \gamma_i/\delta_i $
\STATE  $\bx_{i+1} = \bx_{i}+ \alpha_i \bp_i$ 
\STATE  $\br_{i+1} = \br_{i}- \alpha_i \bq_i$ 
\STATE  $\bz_{i+1} = \bz_i- \alpha_i \by_i$ 
\STATE  $\bp_{i+1} = \bz_{i+1}$
\STATE  $\bq_{i+1} = \bA \bz_{i+1}$ 
\STATE  $\tilde \by_{i+1} = \bH \bq_{i+1}$ 
  \FOR{$j = 0,\,\dots,\, i$}
\STATE  $\Phi_{i,j} = \langle \tilde \by_j,  \bq_{i+1} \rangle$ 
\STATE $\beta_{i,j} = \Phi_{i,j}/ \delta_j^{-1} $
  \ENDFOR
\STATE  $\bp_{i+1} -= \sum\limits_{j=0}^{i}\beta_{i,j} \bp_j  $
\STATE  $\bq_{i+1} -= \sum\limits_{j=0}^{i}\beta_{i,j} \bq_j  $
\ENDFOR
\RETURN{$\bx_{i+1}$}
\end{algorithmic}
\end{algorithm}
\end{minipage}

\newpage 
\bibliographystyle{siamplain}
\bibliography{NonSymOrthodir}

\begin{thebibliography}{10}

\bibitem{zbMATH02027915}
{\sc Z.-Z. Bai, G.~H. Golub, and M.~K. Ng}, {\em Hermitian and skew-{Hermitian}
  splitting methods for non-{Hermitian} positive definite linear systems}, SIAM
  J. Matrix Anal. Appl., 24 (2003), pp.~603--626,
  \url{https://doi.org/10.1137/S0895479801395458}.

\bibitem{zbMATH05029264}
{\sc B.~Beckermann, S.~A. Goreinov, and E.~E. Tyrtyshnikov}, {\em Some remarks
  on the {Elman} estimate for {GMRES}}, SIAM J. Matrix Anal. Appl., 27 (2006),
  pp.~772--778, \url{https://doi.org/10.1137/040618849}.

\bibitem{zbMATH07395831}
{\sc M.~Bonazzoli, X.~Claeys, F.~Nataf, and P.-H. Tournier}, {\em Analysis of
  the {SORAS} domain decomposition preconditioner for non-self-adjoint or
  indefinite problems}, J. Sci. Comput., 89 (2021), p.~27,
  \url{https://doi.org/10.1007/s10915-021-01631-8}.
\newblock Id/No 19.

\bibitem{bonazzoli:hal-03882577}
{\sc M.~Bonazzoli, X.~Claeys, F.~Nataf, and P.-H. Tournier}, {\em {How does the
  partition of unity influence SORAS preconditioner?}}
\newblock working paper or preprint, Dec. 2022,
  \url{https://hal.science/hal-03882577}.

\bibitem{arxiv.2103.16703}
{\sc N.~Bootland, V.~Dolean, I.~G. Graham, C.~Ma, and R.~Scheichl}, {\em
  {G}en{EO} coarse spaces for heterogeneous indefinite elliptic problems},
  arXiv preprint arXiv:2103.16703,  (2021),
  \url{https://doi.org/10.48550/ARXIV.2103.16703}.

\bibitem{zbMATH07726053}
{\sc N.~Bootland, V.~Dolean, I.~G. Graham, C.~Ma, and R.~Scheichl}, {\em
  Overlapping {Schwarz} methods with {GenEO} coarse spaces for indefinite and
  nonself-adjoint problems}, IMA J. Numer. Anal., 43 (2023), pp.~1899--1936,
  \url{https://doi.org/10.1093/imanum/drac036}.

\bibitem{cai1989some}
{\sc X.-C. Cai}, {\em Some domain decomposition algorithms for nonselfadjoint
  elliptic and parabolic partial differential equations}, PhD thesis, New York
  University, 1989,
  \url{https://archive.org/details/somedomaindecomp00caix/page/70/mode/2up}.

\bibitem{cai1992comparison}
{\sc X.-C. Cai, W.~D. Gropp, and D.~E. Keyes}, {\em A comparison of some domain
  decomposition algorithms for nonsymmetric elliptic problems}, in Fifth
  International Symposium on Domain Decomposition Methods for Partial
  Differential Equations, Philadelphia, PA, 1992.

\bibitem{zbMATH00036024}
{\sc X.-C. Cai and O.~B. Widlund}, {\em Domain decomposition algorithms for
  indefinite elliptic problems}, SIAM J. Sci. Stat. Comput., 13 (1992),
  pp.~243--258, \url{https://doi.org/10.1137/0913013}.

\bibitem{calvo2016adaptive}
{\sc J.~G. Calvo and O.~B. Widlund}, {\em An adaptive choice of primal
  constraints for {BDDC} domain decomposition algorithms}, Electron. Trans.
  Numer. Anal, 45 (2016), pp.~524--544.

\bibitem{zbMATH01201042}
{\sc T.~F. Chan, E.~Chow, Y.~Saad, and M.~C. Yeung}, {\em Preserving symmetry
  in preconditioned {Krylov} subspace methods}, SIAM J. Sci. Comput., 20
  (1999), pp.~568--581, \url{https://doi.org/10.1137/S1064827596311554}.

\bibitem{zbMATH05080488}
{\sc M.~Eiermann and O.~G. Ernst}, {\em Geometric aspects of the theory of
  {Krylov} subspace methods}, Acta Numerica, 10 (2001), pp.~251--312,
  \url{https://doi.org/10.1017/S0962492901000046}.

\bibitem{zbMATH03831185}
{\sc S.~C. Eisenstat, H.~C. Elman, and M.~H. Schultz}, {\em Variational
  iterative methods for nonsymmetric systems of linear equations}, SIAM J.
  Numer. Anal., 20 (1983), pp.~345--357, \url{https://doi.org/10.1137/0720023}.

\bibitem{elman1982iterative}
{\sc H.~C. Elman}, {\em Iterative methods for large, sparse, nonsymmetric
  systems of linear equations}, PhD thesis, Yale University, 1982.

\bibitem{zbMATH01268410}
{\sc A.~Essai}, {\em Weighted {FOM} and {GMRES} for solving nonsymmetric linear
  systems}, Numer. Algorithms, 18 (1998), pp.~277--292,
  \url{https://doi.org/10.1023/A:1019177600806}.

\bibitem{zbMATH06713479}
{\sc I.~G. Graham, E.~A. Spence, and E.~Vainikko}, {\em Domain decomposition
  preconditioning for high-frequency {Helmholtz} problems with absorption},
  Math. Comput., 86 (2017), pp.~2089--2127,
  \url{https://doi.org/10.1090/mcom/3190}.

\bibitem{zbMATH07248609}
{\sc I.~G. Graham, E.~A. Spence, and J.~Zou}, {\em Domain decomposition with
  local impedance conditions for the {Helmholtz} equation with absorption},
  SIAM J. Numer. Anal., 58 (2020), pp.~2515--2543,
  \url{https://doi.org/10.1137/19M1272512}.

\bibitem{zbMATH06376430}
{\sc S.~G{\"u}ttel and J.~Pestana}, {\em Some observations on weighted
  {GMRES}}, Numer. Algorithms, 67 (2014), pp.~733--752,
  \url{https://doi.org/10.1007/s11075-013-9820-x},
  \url{strathprints.strath.ac.uk/54749/}.

\bibitem{haferssas2017additive}
{\sc R.~Haferssas, P.~Jolivet, and F.~Nataf}, {\em An additive {S}chwarz method
  type theory for {L}ions's algorithm and a symmetrized optimized restricted
  additive {S}chwarz method}, SIAM Journal on Scientific Computing, 39 (2017),
  pp.~A1345--A1365.

\bibitem{MR3043640}
{\sc F.~Hecht}, {\em New development in freefem++}, J. Numer. Math., 20 (2012),
  pp.~251--265, \url{https://freefem.org/}.

\bibitem{johnson1972matrices}
{\sc C.~R. Johnson}, {\em Matrices whose hermitian part is positive definite},
  PhD thesis, California Institute of Technology, 1972.

\bibitem{zbMATH03389168}
{\sc C.~R. Johnson}, {\em An inequality for matrices whose symmetric part is
  positive definite}, Linear Algebra Appl., 6 (1973), pp.~13--18,
  \url{https://doi.org/10.1016/0024-3795(73)90003-7}.

\bibitem{zbMATH03489299}
{\sc C.~R. Johnson}, {\em Inequalities for a complex matrix whose real part is
  positive definite}, Trans. Am. Math. Soc., 212 (1975), pp.~149--154,
  \url{https://doi.org/10.2307/1998618}.

\bibitem{6877513}
{\sc P.~Jolivet, F.~Hecht, F.~Nataf, and C.~Prud'homme}, {\em Scalable domain
  decomposition preconditioners for heterogeneous elliptic problems}, in SC
  '13: Proceedings of the International Conference on High Performance
  Computing, Networking, Storage and Analysis, 2013, pp.~1--11,
  \url{https://doi.org/10.1145/2503210.2503212}.

\bibitem{klawonn2016adaptive}
{\sc A.~Klawonn, M.~Kuhn, and O.~Rheinbach}, {\em Adaptive coarse spaces for
  {FETI-DP} in three dimensions}, SIAM Journal on Scientific Computing, 38
  (2016), pp.~A2880--A2911.

\bibitem{zbMATH01271905}
{\sc A.~Klawonn and G.~Starke}, {\em Block triangular preconditioners for
  nonsymmetric saddle point problems: {Field}-of-values analysis}, Numer.
  Math., 81 (1999), pp.~577--594, \url{https://doi.org/10.1007/s002110050405}.

\bibitem{zbMATH06385506}
{\sc J.~Liesen and Z.~Strako{\v{s}}}, {\em Krylov subspace methods.
  {Principles} and analysis}, Numer. Math. Sci. Comput., Oxford: Oxford
  University Press, reprint of the 2013 hardback edition~ed., 2015,
  \url{https://doi.org/10.1093/acprof:oso/9780199655410.001.0001}.

\bibitem{arxiv12115969}
{\sc J.~Liesen and P.~Tich{\`y}}, {\em The field of values bound on ideal
  {GMRES}}, arXiv preprint arXiv:1211.5969,  (2012),
  \url{https://doi.org/10.48550/ARXIV.1211.5969}.

\bibitem{zbMATH06394941}
{\sc G.~Meurant and J.~D. Tebbens}, {\em The role eigenvalues play in forming
  {GMRES} residual norms with non-normal matrices}, Numer. Algorithms, 68
  (2015), pp.~143--165, \url{https://doi.org/10.1007/s11075-014-9891-3}.

\bibitem{zbMATH00089377}
{\sc N.~M. Nachtigal, S.~C. Reddy, and L.~N. Trefethen}, {\em How fast are
  nonsymmetric matrix iterations?}, SIAM J. Matrix Anal. Appl., 13 (1992),
  pp.~778--795, \url{https://doi.org/10.1137/0613049}.

\bibitem{pechstein2017unified}
{\sc C.~Pechstein and C.~R. Dohrmann}, {\em A unified framework for adaptive
  {BDDC}}, Electron. Trans. Numer. Anal, 46 (2017), p.~3.

\bibitem{zbMATH06290704}
{\sc J.~Pestana and A.~J. Wathen}, {\em On the choice of preconditioner for
  minimum residual methods for non-{Hermitian} matrices}, J. Comput. Appl.
  Math., 249 (2013), pp.~57--68,
  \url{https://doi.org/10.1016/j.cam.2013.02.020}.

\bibitem{zbMATH01953444}
{\sc Y.~Saad}, {\em Iterative methods for sparse linear systems.},
  Philadelphia, PA: SIAM Society for Industrial {and} Applied Mathematics, 2nd
  ed.~ed., 2003.

\bibitem{zbMATH03967793}
{\sc Y.~Saad and M.~H. Schultz}, {\em {GMRES}: {A} generalized minimal residual
  algorithm for solving nonsymmetric linear systems}, SIAM J. Sci. Stat.
  Comput., 7 (1986), pp.~856--869, \url{https://doi.org/10.1137/0907058}.

\bibitem{zbMATH05626642}
{\sc M.~Sarkis and D.~B. Szyld}, {\em Optimal left and right additive {Schwarz}
  preconditioning for minimal residual methods with {Euclidean} and energy
  norms}, Comput. Methods Appl. Mech. Eng., 196 (2007), pp.~1612--1621,
  \url{https://doi.org/10.1016/j.cma.2006.03.027}.

\bibitem{spill2014}
{\sc N.~Spillane}, {\em Robust domain decomposition methods for symmetric
  positive definite problems}, PhD thesis, UPMC, 2014,
  \url{http://www.theses.fr/2014PA066005/document}.

\bibitem{arxiv.2104.00280}
{\sc N.~{Spillane}}, {\em An abstract theory of domain decomposition methods
  with coarse spaces of the gen{EO} family}, arXiv preprint,  (2021),
  \url{https://doi.org/10.48550/ARXIV.2104.00280}.

\bibitem{2011SpillaneCR}
{\sc N.~{Spillane}, V.~Dolean, P.~Hauret, F.~Nataf, C.~Pechstein, and
  R.~Scheichl}, {\em A robust two-level domain decomposition preconditioner for
  systems of {PDE}s}, C. R. Math. Acad. Sci. Paris, 349 (2011), pp.~1255--1259,
  \url{https://doi.org/10.1016/j.crma.2011.10.021}.

\bibitem{spillane2013abstract}
{\sc N.~{Spillane}, V.~Dolean, P.~Hauret, F.~Nataf, C.~Pechstein, and
  R.~Scheichl}, {\em Abstract robust coarse spaces for systems of {PDE}s via
  generalized eigenproblems in the overlaps}, Numer. Math., 126 (2014),
  pp.~741--770, \url{https://doi.org/10.1007/s00211-013-0576-y}.

\bibitem{SPILLANE:2013:FETI_GenEO_IJNME}
{\sc N.~{Spillane} and D.~J. Rixen}, {\em {Automatic spectral coarse spaces for
  robust {FETI} and {BDD} algorithms}}, Int. J. Numer. Meth. Engng., 95 (2013),
  pp.~953--990.

\bibitem{zbMATH01226274}
{\sc G.~Starke}, {\em Multilevel minimal residual methods for nonsymmetric
  elliptical problems}, Numer. Linear Algebra Appl., 3 (1996), pp.~351--367.

\bibitem{zbMATH01096035}
{\sc G.~Starke}, {\em Field-of-values analysis of preconditioned iterative
  methods for nonsymmetric elliptic problems}, Numer. Math., 78 (1997),
  pp.~103--117, \url{https://doi.org/10.1007/s002110050306}.

\bibitem{zbMATH02113718}
{\sc A.~Toselli and O.~Widlund}, {\em Domain decomposition methods --
  algorithms and theory.}, vol.~34 of Springer Ser. Comput. Math., Berlin:
  Springer, 2005.

\bibitem{FFD:Tournier:2019}
{\sc P.-H. Tournier, P.~Jolivet, and F.~Nataf}, {\em {FFDDM}: Freefem domain
  decomposition method}.
\newblock {https://doc.freefem.org/documentation/ffddm/index.html}, 2019.

\bibitem{vinsome1976orthomin}
{\sc P.~K. Vinsome}, {\em Orthomin, an iterative method for solving sparse sets
  of simultaneous linear equations}, in SPE Symposium on Numerical Simulation
  of Reservoir Performance, OnePetro, 1976.

\bibitem{zbMATH03619254}
{\sc O.~Widlund}, {\em A {Lanczos} method for a class of nonsymmetric systems
  of linear equations}, SIAM J. Numer. Anal., 15 (1978), pp.~801--812,
  \url{https://doi.org/10.1137/0715053}.

\bibitem{zbMATH00149253}
{\sc J.~Xu and X.-C. Cai}, {\em A preconditioned {GMRES} method for
  nonsymmetric or indefinite problems}, Math. Comput., 59 (1992), pp.~311--319,
  \url{https://doi.org/10.2307/2153059}.

\end{thebibliography}

\end{document}